\documentclass{article}
\usepackage{amsthm}

\usepackage[T2A]{fontenc}
\usepackage[cp1251]{inputenc}
\usepackage[english]{babel}
\usepackage[tbtags]{amsmath}
\usepackage{amsfonts,amssymb,mathrsfs,amscd}

\usepackage{euscript,textcomp,verbatim,fancyhdr}
\usepackage{latexsym}
\theoremstyle{plain}
\newtheorem{Thm}{Theorem} [section]
\theoremstyle{definition}
  \newtheorem{Rem}[Thm]{Remark}
\newcommand{\RR}{{\mathbb R}}
\newcommand{\ZZ}{{\mathbb Z}}
\newcommand{\QQ}{{\mathbb Q}}
\newcommand{\NN}{{\mathbb N}}

\def\const{{\rm const}}

\def\Scal{{\rm Scal}}
\newcommand{\rd}{{\rm d}}
\newdimen\normalparindent
\parindent\normalparindent
\let\ge\geqslant

\begin{document}

\title
{Superintegrable Bertrand \\ magnetic geodesic flows}
\author{E.\,A.\ Kudryavtseva, S.\,A.\ Podlipaev}
\date{}

\maketitle

UDK 
514.853, % (514.85 - Геометрические методы в общей механике, 514.853 - в динамике)
517.938.5 %(517.938.5 - Топологические вопросы теории динамических систем)
%\UDC{514.853, 517.938.5}

\begin{abstract}
The problem of description of superintegrable systems (i.e., systems with closed trajectories in a certain domain) in the class of rotationally symmetric natural mechanical systems goes back to Bertrand and Darboux. We describe all superintegrable (in a domain of slow motions) systems in the class of rotationally symmetric magnetic geodesic flows. We show that all sufficiently slow motions in a central magnetic field on a two-dimensional manifold of revolution are periodic if and only if the metric has a constant scalar curvature and the magnetic field is homogeneous, i.e. proportional to the area form.

\end{abstract}

\footnotetext[0]{The work was supported by the Programme of the President of RF for support of Leading Scientific Schools (grant No. NSh-6399.2018.1, Agreement No. 075--02--2018-867)
and the Russian Foundation for Basic Research (grant No. 19--01--00775-a).}

\section{Introduction}

We study magnetic geodesic flows invariant under rotations. Such a dynamical system describes the motion of a charged particle in a central {\em magnetic field} on a two-dimensional Riemannian manifold of revolution. We suppose that the magnetic field is not an identical zero. 
We study {\em slow motions}, i.e. motions with sufficiently small positive level of energy (or with a sufficiently small velocity).
From a geometrical point of view, we study {\em magnetic geodesics} of sufficiently small radii of curvature, where the radius of curvature depends on a position on the surface and is proportional to the ratio of the magnetic field to the area form, with a sufficiently small positive proportionality factor.

We are interested in the systems satisfying the following {\em Bertrand condition for slow motions}: any sufficiently slow motion of the charged particle is periodic. On a geometrical language, the Bertrand condition is formulated as follows: all magnetic geodesics of sufficiently small radii of curvature are closed curves.
We will show (Theorem \ref {thm:Bert}) that the Bertrand condition holds if and only if the metric has a constant scalar curvature and the magnetic field is homogeneous, i.e. proportional to the area form.

Let us give a short historical overview. The problem of description of superintegrable systems (i.e., systems with closed trajectories in a certain domain) in the class of rotationally symmetric natural mechanical systems goes back to Bertrand and Darboux. Natural mechanical systems of Bertrand's type were described under several restrictions by Bertrand \cite {Ber}, Darboux \cite {Dar}, Besse \cite {Bes}, Perlick \cite {Per}, Zagryadskii, Kudryavtseva and Fedoseev \cite {ZKF}, Kudryavtseva and Fedoseev \cite {KF1,KF2,KF3} and others (cf. a survey in \cite {ZKF}).

In this paper, we obtain a description (i.e., a rigorous classification) of all {\em magnetic Bertrand systems} --- superintegrable magnetic geodesic flows on two-dimensional configurational manifolds of revolution satisfying the periodicity condition for slow motions. The problem of describing all {\em electromagnetic Bertrand systems} remains open, i.e. superintegrable systems that define motion of a charged particle under the influence of a potential as well as a magnetic force fields on a two-dimensional configurational manifold of revolution, and we plan to solve it in future.

Let us indicate a difference of our approach from the classical approach of Bertrand \cite {Ber} and many of his followers including \cite{ZKF}. The classical approach is based on studying a {\em family of reduced systems} obtained from the initial system via a {\em time change} on the phase trajectories by the angle coordinate $\varphi=\varphi(t)$ (longitude) on the configuration manifold of revolution.
Such a time change $t\to\varphi(t)$ 
%along the phase trajectories 
makes sense (i.e.\ is regular and monotone) for {\em nonsingular} trajectories: when the longitude is monotone in time ($\dot\varphi(t)\ne0$) along the whole trajectory.
The reduced system has one degree of freedom, depends on one parameter (the constant of kinetic momentum) and, for each parameter value, has a nondegenerate equilibrium point of the ``centre'' type (corresponding to a circular solution).
Clearly, a superintegrability will take place in the case of a {\em common isochronicity} of the family of reduced systems, i.e., when the periods of their solutions near the indicated equilibria coincide, and their common period is commensurable with $\pi$. 
%We stress that the indicated time change $t\to\varphi(t)$ along the phase trajectories makes sense (i.e.\ is regular and monotone) only for {\em nonsingular} trajectories: when the longitude is monotone in time ($\dot\varphi(t)\ne0$).
Since, in the classical case (without a magnetic field) almost all solutions are nonsingular, the superintegrability is equivalent to the common isochronicity. But, in the ``magnetic'' case, there are ``many'' singular solutions (e.g. all sufficiently slow motions), and the indicated time change (together with the beautiful condition of a common isochronicity) does not make sense for them. For such solutions, we write the reduced system without any time change, in contrast to the classical approach. Besides that, we follow the classical approach \cite{Ber}.

Let us proceed with precise statements.

\section{Formulation of the main result} \label {sec:slow}

Let $(Q,g)$ be a smooth two-dimensional Riemannian manifold of revolution, $B$ a differential 2-form on $Q$ invariant under rotations.
The 2-form $B$ is, of course, closed ($\rd B=0$) and, thus, defines a {\em magnetic field} on $Q$. Its invariance under rotations means that the magnetic field is {\em central}.

The manifold of revolution $Q$ is diffeomorphic to either a 2-sphere or a 2-disk or 
a 2-torus or an open cylinder: $Q\setminus\mathrm{Fix}(S^1)\approx I\times S^1$ with coordinates $(r,\varphi)$, where\\
$r\in I$ is a {\em latitude}, i.e. a natural parameter on meridians $I\times\{\varphi_0\}$, \\
$\varphi\in\RR/2\pi\ZZ=S^1$ is a {\em longitude}, i.e. an angle coordinate on parallels $\{r_0\}\times S^1$. \\
Here $I=(r_1,r_2)\subseteq\RR$ (in the cases of a sphere, a 2-disk and a cylinder) or $I=S^1=\RR/L\ZZ$, $L>0$ (in the case of a 2-torus).
Without loss of generality, we will assume that $Q$ is diffeomorphic to a cylinder:
$$
Q\approx(r_1,r_2)\times S^1
$$
(the cases of a 2-sphere and a 2-disk reduce to this case by considering $Q$ punctured at the fixed points of the rotation, while the case of a 2-torus $Q\approx S^1\times S^1$ reduces to it by considering a covering cylinder $\tilde Q=\RR\times S^1$). The Riemannian metric of revolution on $Q$ has the form
$$
g= \rd r^2 + f^2(r) \rd\varphi^2, \qquad (r,\varphi) \in Q,
$$
where $f(r)>0$ is a smooth function (the radius of a parallel $\{r\}\times S^1$). The rotationally symmetric 2-form $B$ on $Q$ has the form 
$$
B=b(r)\rd r\wedge \rd\varphi=a'(r)\rd r\wedge \rd\varphi, \qquad (r,\varphi) \in Q,
$$
where $a(r)$ is a function on $(r_1,r_2)$ defined (up to an additive constant) by the condition $a'(r)=b(r)$.

The motion of the charged particle in the magnetic field $B$ on the surface $(Q,g)$ is described by the Hamiltonian system on $T^*Q$, with the Hamilton function and the complex structure
$$
H=\dfrac{p_r^2}{2} + \dfrac{p_\varphi^2}{2f^2(r)}, \quad \omega
= \rd p_r \wedge \rd r + \rd p_\varphi \wedge \rd\varphi + B
= \rd p_r \wedge \rd r + \rd (p_\varphi + a(r)) \wedge \rd\varphi .
$$
Since the Hamilton function $H$ is independ of the angle variable $\varphi$, the system is integrable via an additional first integral 
$$
K=\tilde p_\varphi := p_\varphi + a(r)
$$
(the kinetic momentum). The first integral $K$ is $2\pi$-periodic, i.e., defines a free Hamiltonian action of the circle on $T^*Q$.

If the magnetic field $B=a'(r)\rd r\wedge \rd\varphi$ has no zeros, the change $r\to a=a(r)$ of the latitude $r\in(r_1,r_2)$ is monotone and regular. The corresponding momentum change is $p_r=a'(r)p_a$. We obtain a manifold 
of revolution $Q\simeq(a_1,a_2)\times S^1$ with Riemannian metric
$$
g = \dfrac{\rd a^2}{R(a)} + \dfrac{\rd\varphi^2}{F(a)},
$$
where $R(a(r))=a'(r)^2>0,$ $F(a(r))=1/f^2(r)>0.$ The Hamilton function and the symplectic structure are
\begin{equation} \label {eq:Ham}
H=R(a)\dfrac{p_a^2}{2} + F(a)\dfrac{p_\varphi^2}{2}, \qquad \omega = \rd p_a \wedge \rd a + \rd (p_\varphi+a) \wedge \rd\varphi.
\end{equation}
In the canonical variables $(a,\varphi,p_a,\tilde p_\varphi=K)$, they take the form
\begin{equation} \label {eq:Ham:}
H=R(a)\dfrac{p^2_a}{2} + F(a)\dfrac{(\tilde p_\varphi-a)^2}{2}, \qquad
\omega=\rd p_a \wedge \rd a + \rd \tilde p_\varphi\wedge \rd\varphi,
\end{equation}
and the equations of motion of the charged particle have the form 
\begin{equation} \label {eq:ODE}
\left\{
\begin{array}{l}
\dot a = \dfrac{\partial H}{\partial p_a}
= R (a)p_a, \\
\dot{\varphi}
= \dfrac{\partial H}{\partial \tilde p_\varphi}
= F(a)(\tilde p_\varphi - a),\\
\dot{p}_a = - \dfrac{\partial H}{ \partial a}
= - R'(a)\dfrac{p^2_a}{2}
- F'(a)\dfrac{(\tilde p_\varphi-a)^2}{2}
+ F(a)(\tilde p_\varphi-a), \\
\dot {\tilde p}_\varphi = - \dfrac{\partial H}{\partial\varphi} = 0. \\
\end{array}
\right.
\end{equation}

Let us study an analogue of the {\em Bertrand problem} (cf. e.g. \cite {ZKF}) for {\em slow} (i.e.\ with small velocity) motions of the charged particle in the magnetic field under consideration.

\begin{Thm} [{\cite[Theorem 7.1]{Pod2015}}] \label {thm:Bert}
Suppose that the central magnetic field $B$ on a two-dimensional Riemannian manifold of revolution $(Q,g)$ is different from the identical zero. Suppose that the ``Bertrand condition for slow motions'' holds: any motion with a sufficiently small nonzero velocity is periodic (in detail: there exists a continuous function $f>0$ on the manifold $Q$ 
such that any motion with initial conditions $(q,p)\in T^*Q$, $0<|p|<f(q)$, is periodic). Then:

{\rm(a)} the magnetic field $B$ is {\em homogeneous}, i.e., proportional to the area form $\rd\sigma=f(r)\rd  r\wedge \rd\varphi$ (hence $B$ has no zeros, the motion is described by the Hamiltonian system (\ref {eq:Ham}) and (\ref {eq:ODE}), the 2-forms $B=\rd a\wedge \rd\varphi$ and $\rd\sigma=\dfrac{\rd a\wedge \rd\varphi}{\sqrt{RF}}$ are proportional and $RF\equiv \lambda^2=\const$);

{\rm(b)} the minimal positive period of the solution with the initial condition $(q,p)\in T^*Q$, $0<|p|<f(q)$, continuously depends on the initial condition $(q,p)$ and tends to the minimal positive period of the linearized system at the equilibrium $(q,0)$ as $|p|\to0$;

{\rm(c)} the scalar curvature $\Scal$ of the manifold 
$(Q,g)$ is constant and equals $\Scal=\dfrac{\lambda^2}{F^3}(F''F-2(F')^2)=-\left(\dfrac{\lambda^2} F\right)''=-R''$.
In particular, $R(a)=\dfrac{\lambda^2}{F(a)}=\lambda_1+\lambda_2a-\Scal\dfrac{a^2}{2}$ for some constants $\lambda_1,\lambda_2\in\RR$.

The converse is also true: if the scalar curvature $\Scal$ of the manifold 
$(Q,g)$ is constant, and the magnetic field $B$ is homogeneous and different from the identical zero, then the Bertrand condition for slow motions holds.
\end{Thm}

\section {Derivation of homogeneity of the magnetic field from the Bertrand condition for slow motions}

In this section, we prove items (a) and (b) of Theorem \ref {thm:Bert}.

Suppose that the magnetic field $B$ has no zeros (the general case is studied on Step 6 below).
Thus, the motion is given by the Hamiltonian system (\ref {eq:Ham}) and (\ref {eq:ODE}).
We pass to more convenient canonical variables obtained from the initial variables by a {\em 
guiding-centre transformation} \cite [\S3]{Nei},
\begin{equation} \label {eq:hat:varphi}
h:(a,\varphi,p_a,p_\varphi) \mapsto
(\hat a, \hat\varphi, p_a,p_\varphi), \qquad
\hat a=a+p_\varphi=K, \qquad \hat\varphi=\varphi-p_a.
\end{equation}
In the new variables, the symplectic structure takes a canonical form:
$$
\omega
=\rd p_a \wedge \rd (\hat a-p_\varphi) + \rd \hat a \wedge \rd(\hat\varphi+p_a)
= - \rd p_a \wedge \rd p_\varphi + \rd \hat a \wedge \rd\hat\varphi.
$$

Perform a {\em scale change} of momenta $h_\varepsilon:(\hat a, \hat\varphi, \hat p_a, \hat p_\varphi) \mapsto (\hat a, \hat\varphi, p_a, p_\varphi)$ by the formula
$$
p_a = {\varepsilon} \hat p_a, \qquad
p_\varphi = {\varepsilon} \hat p_\varphi,
$$
where $\varepsilon>0$ is a small parameter (e.g., the modulus of the velocity vector). Then
$$
H
= \varepsilon^2\left(R(\hat a - \varepsilon \hat p_\varphi)\dfrac{\hat p_a^2}{2}
+ F(\hat a - \varepsilon \hat p_\varphi)\dfrac{\hat p_\varphi^2}{2} \right) ,
\qquad
\omega
= \rd \hat a \wedge \rd\hat\varphi - \varepsilon^2 \rd \hat p_a \wedge \rd \hat p_\varphi .
$$
The equations of motion in new variables $(\hat a, \hat\varphi, \hat p_a, \hat p_\varphi)$ (sometimes called ``slow-fast'' variables \cite{Kud} for the given system) have the form
\begin{equation} \label {eq:ODE:eps}
\left\{
\begin{array}{l}
\dot {\hat a} = - \dfrac{\partial H}{\partial\hat \varphi} = 0,\\
\dot {\hat \varphi}
= \dfrac{\partial H}{\partial \hat a}
= \varepsilon^2\left(R'(\hat a - \varepsilon\hat p_\varphi)\dfrac{{\hat p_a}^2}{2}
+ F'(\hat a - \varepsilon\hat p_\varphi)\dfrac{\hat p_\varphi^2}{2}\right),\\
\dot{\hat p}_a = \dfrac{1}{\varepsilon^2} \dfrac{H}{ \partial \hat p_\varphi}
= - \varepsilon R'(\hat a - \varepsilon\hat p_\varphi)\dfrac{{\hat p_a}^2}{2}
- \varepsilon F'(\hat a - \varepsilon\hat p_\varphi)\dfrac{{\hat p_\varphi}^2}{2}
+ F(\hat a - \varepsilon\hat p_\varphi)\hat p_\varphi,\\
\dot{\hat p}_\varphi = - \dfrac{1}{\varepsilon^2} \dfrac{\partial H}{\partial \hat p_a}
= - R (\hat a - \varepsilon\hat p_\varphi)\hat p_a.
\end{array}
\right.
\end{equation}

\begin{Rem} \label {rem:circ}
Let us find all {\em relative equilibria}, i.e., all phase points at which $\rd H$ and $\rd K$ are proportional (so, the corresponding integral curves $\gamma\subset T^*Q$ are {\em circular orbits} and {\em equilibria}, since they satisfy $a \equiv \const$). Due to the equations (\ref {eq:ODE}), we have $p_a \equiv 0,$ $F'(a)\dfrac{(\tilde p_\varphi - a)^2}{2} + F(a)(a-\tilde p_\varphi) = 0.$
The latter equality implies $\tilde p_\varphi \equiv a$ (the case of an equilibrium) or $\dfrac{F'(a)}{F(a)}(\tilde p_\varphi - a) \equiv 2$ (the case of a circular orbit).
In the new variables, for $\varepsilon>0$, we obtain $\hat p_a \equiv 0$ and either $\hat p_\varphi=0$ (in the case of an equilibrium), or $\dfrac{F'(\hat a - \varepsilon\hat p_\varphi)}{F(\hat a - \varepsilon\hat p_\varphi)}\varepsilon\hat p_\varphi \equiv 2$ 
(in the case of a circular orbit).
\end{Rem}

As $\varepsilon\to0$, the system (\ref {eq:ODE:eps}) tends to a well defined {\em limit system} --- the family of harmonic oscillators on each fibre, i.e.\ the family of Hamiltonian systems
\begin{equation} \label {eq:harm}
\left(T^*_{q}Q, \ \omega_{q}=-\rd \hat p_a \wedge \rd \hat p_\varphi, \ H_{q}=R(\hat a) \dfrac{\hat p_a^2}{2} + F(\hat a)\dfrac{\hat p_\varphi^2}{2}\right)
\end{equation}
with parameters $q=(\hat a,\hat\varphi)\in(a_1,a_2)\times S^1=Q,$ where $\dot{\hat a}=\dot{\hat\varphi}=0.$
Under the inverse change $h^{-1}\circ h_0$, any solution of the limit system is transformed to the corresponding equilibrium $h^{-1}(\hat a,\hat\varphi,0,0)=(\hat a,\hat\varphi,0,0)$ of the initial system (\ref {eq:ODE}) indicated in Remark \ref {rem:circ}.

Thus, solutions of the system (\ref {eq:ODE:eps}) correspond to {\em slow} (i.e.\ having a small velocity) motions given by the system (\ref {eq:ODE}) with (\ref {eq:Ham}).
Notice that the limit system has the form $\dot{\hat a}=\dot{\hat\varphi}=0,$ $\dot{\hat p}_a=F(\hat a)\hat p_\varphi,$ $\dot{\hat p}_\varphi = -R (\hat a)\hat p_a.$ Therefore all its solutions, apart from the equilibria $\hat p_a=\hat p_\varphi=0,$ satisfy the equation $\ddot{\hat p}_a = -R(\hat a)F(\hat a)\hat p_a$ and, hence, determine harmonic oscillations with angular frequency $\sqrt{R(\varkappa)F(\varkappa)}$ and minimal positive period
\begin{equation} \label {eq:T}
T(\varkappa)=\dfrac{2\pi}{\sqrt{R(\varkappa)F(\varkappa)}},
\end{equation}
where $\varkappa\in(a_1,a_2)$ is the value of the first integral $K=\hat a=\tilde p_\varphi$ on the given solution.

\begin{proof}[Proof of items (a) and (b) of Theorem \ref {thm:Bert}]
Step 1. Let us first suppose that the magnetic field $B$ has no zeros.

Fix a number $\varkappa\in(a_1,a_2),$ and consider the {\em effective potential}
$$
U_{\varkappa}=U_{\varkappa}(a):=F(a)\dfrac{(\varkappa - a)^2}{2}, \qquad a\in(a_1,a_2),
$$
of the system (\ref {eq:Ham}). The Hamilton function has the form $H=R(a)\dfrac{p_a^2}{2} + U_{K} (a)$.
In the new variables, the effective potential is
\begin{equation} \label {eq:Ueff}
\dfrac1{\varepsilon^2}U_{\varkappa}(a)
=F(\varkappa - \varepsilon\hat p_\varphi)\dfrac{\hat p_\varphi^2}{2}
=:\Hat U_{\varepsilon,\varkappa}(\hat p_\varphi),
\qquad \hat p_\varphi\in\RR,
\end{equation}
and the Hamilton function is $\dfrac1{\varepsilon^2}H=R(K - \varepsilon\hat p_\varphi)\dfrac{\hat p_a^2}{2} + \Hat U_{\varepsilon,K} (\hat p_\varphi).$

Let us study the {\em reduced system} corresponding to the $2\pi$-periodic first integral $K$, with the Hamilton function and the symplectic structure 
\begin{equation} \label {eq:reduc}
H_{\varepsilon,\varkappa}
:= \dfrac1{\varepsilon^2}H|_{\{K=\varkappa\}}
= R(\varkappa - \varepsilon\hat p_\varphi)\dfrac{\hat p_a^2}{2} + \Hat U_{\varepsilon,\varkappa} (\hat p_\varphi), \qquad
- \rd \hat p_a \wedge \rd \hat p_\varphi,
\end{equation}
where $\varkappa\in(a_1,a_2)$ is the parameter of the reduced system, $0<\varepsilon\ll1$ is a small (``scale'') parameter. The reduced system is defined on a ``reduced phase plane'' $\{K=\varkappa\}/S^1\subset(T^*Q)/S^1$ with phase variables $(\hat p_a,\hat p_\varphi)$ and an induced symplectic structure, where one considers the Hamiltonian action of the circle $S^1$ on $T^*Q$ generated by the $2\pi$-periodic first integral $K$.

The Hamilton function $H_{\varepsilon,\varkappa}$ of the reduced system (\ref {eq:reduc}) equals the sum of the ``reduced kinetic energy'' $R(\varkappa - \varepsilon\hat p_\varphi)\dfrac{\hat p_a^2}{2}$ (quadratic in the ``reduced momentum'' $\hat p_a$) and the effective potential $\Hat U_{\varepsilon,\varkappa}(\hat p_\varphi)$ (depending on the ``reduced coordinate'' $\hat p_\varphi$ only). That is, the reduced system is a natural mechanical system with one degree of freedom, therefore we can explicitly solve it by standard techniques. Let us do this.

For the value $\Hat E$ of the reduced Hamilton function $H_{\varepsilon,\varkappa}$, we have $R(\varkappa-\varepsilon\hat p_\varphi) \hat p_a^2 = 2\Hat E-2\Hat U_{\varepsilon,\varkappa}(\hat p_\varphi)\ge0.$ Therefore the value $\hat p_a$ can be expressed in terms of $\varepsilon,\Hat E,\varkappa$ and $\hat p_\varphi$ by the formula
\begin{equation} \label {eq:p}
\hat p_a
=\pm\sqrt{\dfrac{2\Hat E-2\Hat U_{\varepsilon,\varkappa}(\hat p_\varphi)}{R(\varkappa-\varepsilon\hat p_\varphi)}}
\stackrel{(\ref {eq:Ueff})}
= \pm\sqrt{\dfrac{2\Hat E- F(\varkappa - \varepsilon\hat p_\varphi) \hat p_\varphi^2}{R(\varkappa-\varepsilon\hat p_\varphi)}}.
\end{equation}

Let us fix real numbers $\varkappa\in(a_1,a_2)$, $\Hat E\in(0,1)$ and $\hat p_\varphi\in\RR$ such that $|\hat p_\varphi|^2<\dfrac{2\Hat E}{F(\varkappa)}$. In the limit $\varepsilon\to0$, we obtain
\begin{equation} \label {eq:p:lim}
\Hat U_{0,\varkappa}(\hat p_\varphi)=F(\varkappa)\dfrac{\hat p_\varphi^2}{2}, \qquad
\hat p_a=\pm\sqrt{\dfrac{2\Hat E-F(\varkappa)\hat p_\varphi^2}{R(\varkappa)}}.
\end{equation}
Notice that the numerator of the radical expression in (\ref {eq:p:lim}) has two simple roots $\hat p_\varphi = A_\pm(0,\Hat E,\varkappa):=\pm\sqrt{\dfrac{2\Hat E}{F(\varkappa)}},$ while the denominator is positive everywhere. 
It follows from the Inverse Function Theorem that, for sufficiently small ``perturbation'' $0<\varepsilon\ll1$, the numerator of the radical expression in (\ref {eq:p}) also has two simple roots, denoted by $A_\pm(\varepsilon,\Hat E,\varkappa)$, that are $O(\varepsilon)$-close to the indicated roots.

This implies (cf. e.g. \cite[\S4, Proposition 2]{ZKF}) that, for any $0<\varepsilon\ll1$, there exists a (unique up to time shifts $t\mapsto t+t_0$) solution
\begin{equation} \label {eq:sol:reduc}
\bar\gamma_{\varepsilon;\Hat E,\varkappa}(t) = ((\hat p_\varphi)_{\varepsilon;\Hat E,\varkappa}(t),(\hat p_a)_{\varepsilon;\Hat E,\varkappa}(t)), \qquad t\in\RR,
\end{equation}
of the reduced system (\ref {eq:reduc}) on the given level $\Hat E$ of the Hamiltonian function, on which the variable $\hat p_\varphi=(\hat p_\varphi)_{\varepsilon,\Hat E,\varkappa}(t)$ takes at least one value from the segment $[A_-(\varepsilon,\Hat E,\varkappa), A_+(\varepsilon,\Hat E,\varkappa)].$ Actually, on the solution (\ref {eq:sol:reduc}), the variable $\hat p_\varphi$ takes all values from the indicated segment, i.e. $(\hat p_\varphi)_{\varepsilon,\Hat E,\varkappa}(\RR^1)$ coincides with this segment.

This implies (cf. e.g. \cite[\S4, предложение 3]{ZKF}) that the time-dependence of the variable $\hat p_\varphi=(\hat p_\varphi)_{\varepsilon,\Hat E,\varkappa}(t)$ and, hence, of the {\em latitude} $a=\varkappa-\varepsilon(\hat p_\varphi)_{\varepsilon,\Hat E,\varkappa}(t)$ is a periodic function, whose half-period equals the time $T(\varepsilon,\hat E,\varkappa)/2=t_+-t_-$ of motion between its adjacent minimum and maximum, where $(\hat p_\varphi)_{\varepsilon,\Hat E,\varkappa}(t_-)=A_-(\varepsilon,\Hat E,\varkappa),$ $(\hat p_\varphi)_{\varepsilon,\Hat E,\varkappa}(t_+)=A_+(\varepsilon,\Hat E,\varkappa).$

Step 2. Now, let us ``lift'' the solution (\ref {eq:sol:reduc}) of the reduced system to the phase space $T^*Q$, i.e., consider the corresponding solution
$$
\gamma_{\varepsilon;\Hat E,\varkappa}(t)
 = (\varkappa,
\hat\varphi_{\varepsilon;\Hat E,\varkappa}(t),
(\hat p_a)_{\varepsilon;\Hat E,\varkappa}(t),
(\hat p_\varphi)_{\varepsilon;\Hat E,\varkappa}(t)), \qquad t\in\RR,
$$
of the initial system (\ref {eq:ODE:eps}) (such a solution is unique up to shifts of the angle variable $\varphi\mapsto\varphi+\varphi_0$).

Let us project the corresponding phase curve $\{\gamma_{\varepsilon;\Hat E,\varkappa}(t)\mid t\in\RR\}$ to the configuration cylinder $Q$ with coordinates $(a,\hat\varphi)\in(a_1,a_2)\times S^1$. We obtain the corresponding ``orbit'' of the charged particle \footnote{When we use here the term ``orbit'', we abuse notation. Strictly speaking, on the configuration cylinder $Q$, we introduced the coordinates $(a,\varphi)$ rather than $(a,\hat\varphi)=(a,\varphi-p_a)=(a,\varphi-\varepsilon\hat p_a)$. Thus, the true orbit of the charged particle is obtained via a projection onto the configuration cylinder $Q$ with coordinates $(a,\varphi)$, hence it is only $O(\varepsilon)$-close to the ``orbit'' mentioned above (and does not coincide with it, generally speaking). We remark that, for the true orbit, the pericentres, apocentres and the time of motion between adjacent pericentre and apocentre are exactly the same as for our ``orbit''. This similarity between our ``orbit'' and the true orbit will suffice for us to prove Theorem \ref {thm:Bert}.}:
$$
\{(a,\hat\varphi)
=(\varkappa-\varepsilon(\hat p_\varphi)_{\varepsilon,\Hat E,\varkappa}(t),
\hat \varphi_{\varepsilon,\Hat E,\varkappa}(t))
\mid t\in\RR\} \subset Q.
$$
On this ``orbit'', consider the points (called the {\em pericentres} and {\em apocentres} of the orbit), whose {\em latitudes} $a=\varkappa-\varepsilon(\hat p_\varphi)_{\varepsilon,\Hat E,\varkappa}(t)$ are the (left and right, respectively) endpoints of the corresponding segment $[\varkappa-\varepsilon A_+(\varepsilon,\Hat E,\varkappa),\ \varkappa-\varepsilon A_-(\varepsilon,\Hat E,\varkappa)].$

Step 3. As above, let us fix the numbers $\varkappa\in(a_1,a_2),$ $\Hat E\in(0,1)$ and $\hat p_\varphi\in\RR$ such that $|\hat p_\varphi|^2<\dfrac{2\Hat E}{F(\varkappa)}.$ Following \cite {ZKF}, let us compute the minimal positive period $T(\varepsilon, \hat{E}, \varkappa)=2(t_+-t_-)$ of the latitude function $a=\varkappa-\varepsilon(\hat p_\varphi)_{\varepsilon,\Hat E,\varkappa}(t)$, when $0<\varepsilon\ll1$:
$$
T(\varepsilon, \hat{E}, \varkappa)
= 2 \int\limits_{A_-(\varepsilon,\Hat E,\varkappa)}^{A_+(\varepsilon,\Hat E,\varkappa)}
\dfrac{\rd \hat p_\varphi}{{(\dot {\hat p}_\varphi)}_{\varepsilon,\Hat E,\varkappa}}
\stackrel{(\ref {eq:ODE:eps}),(\ref {eq:p})}=
$$
$$
= 2 \int\limits_{A_-(\varepsilon,\Hat E,\varkappa)}^{A_+(\varepsilon,\Hat E,\varkappa)}
\dfrac{\rd u}{\sqrt{R (\varkappa - \varepsilon u)} \sqrt{2\Hat E- F(\varkappa - \varepsilon u)u^2}}.
$$
In particular, $T(0,\hat{E},\varkappa)=\dfrac{2\pi}{\sqrt{R(\varkappa)F(\varkappa)}}=T(\varkappa),$ cf.\ (\ref {eq:T}).

Step 4. Consider the real number
$$
\Phi(\varepsilon, \Hat{E}, \varkappa)
:=\varphi_{\varepsilon,\Hat E,\varkappa}(t_+)-\varphi_{\varepsilon,\Hat E,\varkappa}(t_-)
\stackrel{(\ref {eq:hat:varphi}),(\ref {eq:p})}
=\hat\varphi_{\varepsilon,\Hat E,\varkappa}(t_+)-\hat\varphi_{\varepsilon,\Hat E,\varkappa}(t_-),
$$
i.e.\ the difference of {\em longitudes} of the adjacent pericentre and apocentre of the ``orbit'' (cf. Step 2).
Following \cite {ZKF}, for each $\varepsilon\in\RR$, $0<|\varepsilon|\ll1,$ we have
$$
\dfrac1{\varepsilon^2}\Phi(\varepsilon, \hat{E}, \varkappa)
= 2 \int\limits_{A_-(\varepsilon,\Hat E,\varkappa)}^{A_+(\varepsilon,\Hat E,\varkappa)}
\dfrac{\dot{\hat{\varphi}}_{\varepsilon,\Hat E,\varkappa} \rd\hat p_\varphi}
{(\dot{\hat p}_\varphi)_{\varepsilon,\Hat E,\varkappa}}
\stackrel{(\ref {eq:ODE:eps}),(\ref {eq:p})}=
$$
$$
= \sqrt2 \int\limits_{A_-(\varepsilon,\Hat E,\varkappa)}^{A_+(\varepsilon,\Hat E,\varkappa)}
\left(
\dfrac{R'(\varkappa - \varepsilon u)
\sqrt{\Hat E- F(\varkappa - \varepsilon u)\dfrac{u^2}2}} {R(\varkappa-\varepsilon u)^{3/2}} +
\right.
$$
$$
\left.
+ \dfrac{F'(\varkappa - \varepsilon u)\dfrac{u^2}{2}}
{\sqrt{R (\varkappa - \varepsilon u)} \sqrt{\Hat E
- F(\varkappa - \varepsilon u)\dfrac{u^2}2}}
\right) \rd u = O(1).
$$

Now suppose that the Bertrand condition for slow motions holds, i.e., the solution $\gamma_{\varepsilon,\Hat E,\varkappa}(t)$ from Step 2 is periodic, when $0<\varepsilon\ll1$. Therefore, its minimal positive period is a multiple of the minimal positive period $T(\varepsilon,\hat E,\varkappa)$ of the latitude function $a=\varkappa-\varepsilon(\hat p_\varphi)_{\varepsilon,\Hat E,K}(t)$, i.e.\ has the form $k T(\varepsilon,\hat E,\varkappa)$ for some $k\in\NN.$ Therefore, the increment $k\Phi(\varepsilon, \Hat{E}, \varkappa)$ of the ``longitude'' $\hat\varphi_{\varepsilon,\Hat E,K}(t)$ in time $k T(\varepsilon,\hat E,\varkappa)$ is a multiple of $2\pi,$ i.e.\ has the form $k\Phi(\varepsilon, \hat{E}, \varkappa)=2\pi\ell$ for some $\ell\in\NN.$ Here the integers $k,\ell$ depend, generally speaking, on $\varepsilon, \hat{E}, \varkappa$.

Thus, the number $\dfrac{1}{2\pi}\Phi(\varepsilon, \hat{E}, \varkappa)=\frac\ell k$ is rational for any $\varepsilon, \hat{E}, \varkappa$ such that $\Hat E\in(0,1)$ is fixed and $|\varepsilon|>0$ is small enough. Since the function $\Phi(\varepsilon, \hat{E}, \varkappa)$ is continuous (due to the fact that $A_{\pm}(\varepsilon,\Hat E,\varkappa)$ are simple roots of the equation $\Hat U_{\varepsilon,\varkappa}(\hat p_\varphi)=\Hat E$ by Step 1) and takes values in a discrete set $\pi\QQ,$ it must be constant. Since it has order $O(\varepsilon^2)$, it tends to 0 as $\varepsilon\to0$, hence it is identically equal to 0.

Since $\Phi(\varepsilon, \Hat{E}, \varkappa)\equiv0,$ we conclude that the ``longitude'' function $\hat\varphi_{\varepsilon,\Hat E,\varkappa}(t)$ itself (rather than just its time derivative) is $T(\varepsilon, \hat{E}, \varkappa)$-periodic, similarly to the latitude function $a=\varkappa-\varepsilon(\hat p_\varphi)_{\varepsilon,\Hat E,K}(t)$. Therefore the solution $\gamma_{\varepsilon,\Hat E,\varkappa}(t)$ is also $T(\varepsilon, \hat{E}, \varkappa)$-periodic. We obtain from Step 3 that, for any fixed $\Hat E\in(0,1)$ and $0<\varepsilon\ll1$,
$$
T(\varepsilon,\hat{E},\varkappa)-T(\varkappa)=T(\varepsilon,\hat{E},\varkappa)-T(0,\hat{E},\varkappa)=O(\varepsilon),
$$
hence item (b) is proved.

Step 5. Let us prove item (a). By Step 4, we have $\Phi(\varepsilon, \Hat{E}, \varkappa)=0$ for any $0<|\varepsilon|\ll1$, whence $\lim_{\varepsilon\to0}\dfrac{1}{\varepsilon^2}\Phi(\varepsilon, \Hat{E}, \varkappa)=0$. On the other hand, taking into account formulae from Step 4, it is not hard to compute
$$
\lim_{\varepsilon\to0}\dfrac{1}{\varepsilon^2}\Phi(\varepsilon, \Hat{E}, \varkappa)
= 2 \int\limits_{A_-(0,\Hat E,\varkappa)}^{A_+(0,\Hat E,\varkappa)}
\dfrac{\dot{\hat\varphi}_{0,\Hat E,\varkappa} \rd \hat p_\varphi}{(\dot{\hat p}_\varphi)_{0,\Hat E,\varkappa}}
$$
$$
= \sqrt2 \int\limits_{A_-(0,\Hat E,\varkappa)}^{A_+(0,\Hat E,\varkappa)}
\left(
\dfrac{R'(\varkappa)} {R(\varkappa)^{3/2}}
\sqrt{\Hat E - F(\varkappa)\dfrac{u^2}2} + \dfrac{F'(\varkappa)\dfrac{u^2}{2}}
{\sqrt{R (\varkappa)} \sqrt{\Hat E - F(\varkappa)\dfrac{u^2}2}} \right) \rd u
$$
$$
= \pi \hat{E}
\dfrac{R'(\varkappa)F(\varkappa) + R(\varkappa)F'(\varkappa)}{(R(\varkappa)F(\varkappa))^\frac{3}{2}}
= - 2\pi \Hat E
\dfrac{\rd}{\rd \varkappa} \left( \dfrac{1}{\sqrt{R(\varkappa)F(\varkappa)}} \right).
$$

Due to the latter formula and the identity $\Phi(\varepsilon, \hat{E}, \varkappa) \equiv 0$ proved for any $\Hat E\in(0,1)$, $\varkappa\in(a_1,a_2)$ and $0<\varepsilon\ll1$, we obtain that the function $\dfrac{1}{\sqrt{R(\varkappa)F(\varkappa)}}$ is constant, whence $RF=\const$ (i.e., the magnetic field $B$ is homogeneous).

Step 6. Let us prove (a) in the general situation, i.e.\ without the assumption that the magnetic field $B$ has no zeros. Let $Z$ be the zero locus of the field $B$. By assumption, $Q\setminus Z\ne\varnothing$. We showed on Steps 1---5 that the magnetic field $B$ is homogeneous on each connected component $Q_i$ of the open set $Q\setminus Z$. This implies that $B$ has no zeros on $\overline{Q_i}$ (since the area form $\rd\sigma$ has no zeros). On the other hand, $B=0$ on $Z$. Therefore, $\overline{Q_i}\cap Z=\varnothing$, whence $Q_i=\overline{Q_i}$, and hence $Q_i=Q$. Thus, $B$ is homogeneous on the whole configuration manifold $Q=Q_i$, as required.
\end{proof}

\section {Derivation of constancy of scalar curvature from the Bertrand condition for slow motions}

Here we finish the proof of Theorem \ref {thm:Bert} by proving its item (c).

First, we will give a sketch of the proof of item (c) of Theorem \ref {thm:Bert}.

Following \cite[\S4, proof of Proposition 4]{ZKF}, let us expand the value $\Phi(E, K):=\Phi(1,E,K)$, which is the difference of longitudes at the adjacent pericentres of the orbit, into a power series in the small parameter $h>0$, and let the coefficients at the lower powers $h^0, h^2, h^4$ be equal to 0. Here the functions $E=E(c,h)$ and $K=K(c,h)$ are determined by the conditions $a_-(E, K)=c-h,$ $a_+(E, K)=c+h,$ where $c\in(a_1,a_2),$ $a_\pm(E,K):=K-A_\mp(1,E,K)$. Actually, the coefficients at $h^0$ and $h^2$ equal 0 due to Steps 4 and 5 (respectively) from the previous section.

In more detail: taking into account (\ref {eq:ODE}) and $RF\equiv \lambda^2=\const$ (by item (а) of Theorem \ref {thm:Bert}), we have
$$
\Phi(E, K)
= 2\int\limits_{a_-(E, K)}^{a_+(E, K)} \dfrac{\dot{\varphi}}{\dot{a}}\rd a
= 2 \int\limits_{a_-(E, K)}^{a_+(E, K)} \dfrac{F(a)(K-a)}{\sqrt{R(a)}\sqrt{2E-F(a)(K-a)^2}}\rd a
$$
$$
= 2 \int\limits_{a_-(E, K)}^{a_+(E, K)} \dfrac{F^\frac{3}{2}(a)(K-a)}{|\lambda|\sqrt{2E-F(a)(K-a)^2}}\rd a .
$$
Therefore, after the change $a = c + ht, a_- = c-h, a_+ = c+h, \rd a = h\rd t$, we obtain
$$
-\dfrac{|\lambda|}{2F(c)}\, \Phi(E(c,h), K(c,h)) =
$$
$$
= \int\limits_{-1}^1 \left(\dfrac{th}{\sqrt{1-t^2}}
+ \dfrac{(1-2t^2)(\hat F_1+\hat F_2th)}{2\sqrt{1-t^2}}h^2
+ \dfrac{3\hat F_1^3 -6\hat F_1\hat F_2 + 4\hat F_3(1 + t^2 - 2t^4)}{8\sqrt{1-t^2}}h^4
\right)\rd t
$$
$$
+ \bar{\bar{o}}(h^4)
= \dfrac{\pi}{8} \left(3\hat F_1^3 -6\hat F_1\hat F_2 + 4\hat F_3(1 + 1/2 - 3/4)\right) h^4 + \bar{\bar{o}}(h^4) =
$$
$$
= \dfrac{\pi}{8} \left(3\hat F_1^3 -6\hat F_1\hat F_2 + 3\hat F_3\right) h^4 + \bar{\bar{o}}(h^4),
$$
where $\hat{F_i} := \dfrac{F_i(c)}{F(c)},$ and $F_i(c)$ is the coefficient at $(a-c)^i$ in the Taylor expansion of the function $F=F(a)$ at the point $c$. Since $\Phi(E, K)\equiv0$ by Step 4 of the proof of items (a,b) of Theorem \ref {thm:Bert}, we conclude that the coefficient at $h^4$ must be 0, i.e.
\begin{equation} \label {eq:2nd}
6F'^3 - 6FF'F'' + F^2F''' \equiv 0.
\end{equation}

Let us show that the condition (\ref {eq:2nd}) is equivalent to the constancy of the scalar curvature. On the manifold of revolution with the Riemannian metric $g = dr^2 + f^2(r)d\varphi^2$, the scalar curvature equals $\Scal=-2\dfrac{f''(r)}{f(r)},$ cf.\ \cite[\S1.2, Remark 4]{ZKF}. Taking into account the change $a=a(r),$ $\dfrac{da}{dr}=\lambda f(r)=\dfrac{\lambda}{\sqrt{F(a)}},$ we obtain 
$$
\Scal=\dfrac{\lambda^2}{F^3(a)}(F''(a)F(a)-2F'(a)^2)=-\left(\dfrac{\lambda^2}{F(a)}\right)''.
$$
One easily checks that the condition $\Scal'\equiv0$ is equivalent to the condition (\ref {eq:2nd}).

Thus, the scalar curvature $\Scal\equiv\const$, and item (c) is proved.

It remains to prove that the conditions (a,b) in Theorem \ref {thm:Bert} are not only necessary, but also sufficient to fulfill the Bertrand condition for slow motions.
Let $B=\lambda \rd\sigma$, where $\rd\sigma$ is the area form, $\lambda=\lambda(a,\varphi)$ is a smooth function (not necessarily constant). Then the solutions $\gamma(t)=(a(t),\varphi(t))$ of the corresponding Lagrange system of equations with energy level $\dfrac{\varepsilon^2}2>0$ are exactly smooth parametrized curves $\gamma=\gamma(t)$ with velocity $|\dot\gamma(t)|=\varepsilon$ and covariant acceleration $\lambda(\gamma(t))\varepsilon.$ Therefore $s\mapsto\gamma(s/\varepsilon)$ 
is a naturally parametrized curve (with a natural parameter $s$) with geodesic curvature  $\dfrac{\lambda(\gamma(s/\varepsilon))}{\varepsilon}.$
Therefore, if the conditions (a,b) of Theorem \ref{thm:Bert} hold, then the orbits of all sufficiently slow motions of the charged particle will be circles and, in particular, will be closed. Hence, the Bertrand condition for slow motions will be fulfilled.

Theorem \ref {thm:Bert} is completely proved. \qed

The authors are thankful to A.~I.~Neishtadt for indicating a construction of slow-fast variables in the problem on the motion of a charged particle in a slow magnetic field by means of a guiding-centre transformation \cite {Nei}, to A.~A.~Oshemkov for the idea of a possible relation of the condition (\ref {eq:2nd}) with the constancy of the scalar curvature, to A.~Albouy for discussing the method of common isochronicity of a family of systems for solving the Bertrand problem.

E.A.\ Kudryavtseva\\
Lomonosov Moscow State University\\
{\em E-mail:} eakudr at mech.math.msu.su

\medskip
S.A.\ Podlipaev\\
Lomonosov Moscow State University\\
{\em E-mail:} podlipaev.sergey at gmail.com

\end{document}